\documentclass[a4paper, 11pt, oneside]{amsart}

\usepackage[english]{babel}
\usepackage{amsthm}
\usepackage{amsmath}
\usepackage{amssymb}
\usepackage[ansinew]{inputenc}
\usepackage{enumerate}
\usepackage{graphicx}
\usepackage[english,intoc]{nomencl}
\usepackage{fancyhdr}
\usepackage[titletoc]{appendix}
\usepackage[Lenny]{fncychap}
\usepackage{xcolor}   
\usepackage{hyperref}
\hypersetup{
    colorlinks=false, 
    linktoc=page,     
    linkbordercolor=blue,  
    pdfborderstyle={/S/U/W 1}    }   
\usepackage{geometry}
\usepackage{tikz}
\usetikzlibrary{arrows}

\DeclareMathOperator{\Tr}{Tr}
\DeclareMathOperator{\id}{id}
\DeclareMathOperator{\spann}{span}
\DeclareMathOperator{\Hom}{Hom}

\DeclareMathOperator{\Aut}{Aut}
\DeclareMathOperator{\Pol}{Pol}
\DeclareMathOperator{\supp}{supp}

\title{A note on reduced and von Neumann algebraic free wreath products}
\author{Jonas Wahl*}
\thanks{*KU~Leuven, Department of Mathematics, Leuven (Belgium), jonas.wahl@wis.kuleuven.be \\
    Partly supported by ERC Advanced Grant Non-Commutative Distributions in Free Prabability and ERC Consolidator Grant 614195 from the European Research Council under the European Union's Seventh Framework Programme.}

\newtheorem*{thm*}{Theorem}
\newtheorem*{prop*}{Proposition}
\newtheorem*{lem*}{Lemma}
\newtheorem*{cor*}{Corollary}
\newtheorem{thmx}{Theorem}
\newtheorem{thm}{Theorem}[section]
\newtheorem{lem}[thm]{Lemma}
\newtheorem{cor}[thm]{Corollary}
\newtheorem{prop}[thm]{Proposition}

\theoremstyle{definition}
\newtheorem{case}{Case}
\newtheorem{definition}[thm]{Definition}
\newtheorem{rem}[thm]{Remark}

\newtheorem{notation}[thm]{Notation}
\newtheorem*{ack}{Acknowledgements}

\begin{document}

\begin{abstract}
In this paper, we study operator algebraic properties of the reduced and von Neumann algebraic versions of the free wreath products $\mathbb G \wr_* S_N^+$, where $\mathbb G$ is a compact matrix quantum group. Based on recent result on their corepresentation theory by Lemeux and Tarrago in \cite{LemTa}, we prove that $\mathbb G \wr_* S_N^+$ is of Kac type whenever $\mathbb G$ is, and that the reduced version of $\mathbb G \wr_* S_N^+$ is simple with unique trace state whenever $N \geq 8$. Moreover, we prove that the reduced von Neumann algebra of $\mathbb G \wr_* S_N^+$ does not have property $\Gamma$.
\end{abstract}
\maketitle
\section*{Introduction}
Following the introduction of \emph{compact matrix quantum groups (CMQGs)} by Woronowicz in \cite{Wor87}, many fascinating examples have been discovered and studied from different points of view. Many CMQGs have a rich combinatorial structure encoded in their corepresentation theory (see for instance \cite{BanSp}), and many give rise to new examples of $C^*$- and von Neumann algebras, some of which have been analysed in the works of Banica (\cite{Ban97}), Vaes and Vergnioux (\cite{VaVer}), Brannan (\cite{Bra1}, \cite{Bra2}) and others. In this paper, we deal with a class of CMQGs called \emph{free wreath products} which has been introduced by Bichon in \cite{Bic04}. This article is an immediate follow-up to the works of Lemeux and Tarrago (\cite{Lem2}, \cite{LemTa}). \\
Consider a finite directed graph $\mathcal G = (V,E)$ consisting of a finite set of vertices $V$ and a set of edges $E \subset V \times V$ and let $\Aut(\mathcal G)$ denote its automorphism group, i.e. the group of bijections $\sigma: V \to V$ such that $(v,w) \in E$ if and only if $(\sigma(v),\sigma(w)) \in E$. It is a very natural question to ask if and how one can deduce the automorphism group $\Aut(\mathcal G^{\sqcup N})$ of the disjoint union $\mathcal G^{\sqcup N} = (V^{\sqcup N}, E^{\sqcup N}) $
of $N$ copies of $\mathcal G$ from  the original automorphism group $\Aut(\mathcal G)$. If the graph $\mathcal G$ is connected, the relation between $\Aut(\mathcal G^{\sqcup N})$ and $\Aut(\mathcal G)$ can be nicely described using a well known group construction called the (classical) wreath product: \\
If $G$ is a group and $S_N$ denotes the permutation group acting on $N$ points $( N \in \mathbb N_{\geq 1})$, the \emph{wreath product} $G \wr S_N$ is defined as the semidirect product $G^N \rtimes_{\varphi} S_N$ where
\[ \varphi: S_N \to \Aut(G^N), \ \varphi(\sigma)(g_1,\dots,g_N) = (g_{\sigma(1)}, \dots, g_{\sigma(N)}).  \]
Using this construction, one has 
\[ \Aut(\mathcal G^{\sqcup N}) = \Aut(\mathcal G) \wr S_N \]
for a finite connected graph $\mathcal G$. For example, the automorphism group of the graph

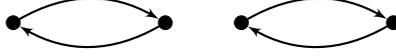
\begin{center}
\begin{tikzpicture}
\tikzset{vertex/.style = {draw, shape=circle,fill=black,scale=0.5}}
\tikzset{edge/.style = {->,> = latex', thick}}
\node[vertex] (a) at (3,0) {};
\node[vertex] (b) at (5,0) {};
\node[vertex] (c) at (6,0) {};
\node[vertex] (d) at (8,0) {};
\draw[edge] (a)  to[bend left] (b);
\draw[edge] (b)  to[bend left] (a);
\draw[edge] (c)  to[bend left] (d);
\draw[edge] (d)  to[bend left] (c);
\end{tikzpicture}
\end{center}
is $\mathbb Z/ 2 \mathbb Z \wr S_2$ and, more general, the isometry group of a hypercube in $\mathbb R^N$ is $\mathbb Z/ 2 \mathbb Z \wr S_N$. \\
In \cite{Bic03}, Bichon introduced a quantum group analogue of the automorphism group of a finite graph $\mathcal G$, say $A_{aut}(\mathcal G)$, and in \cite{Bic04} he constructed a free wreath product $\wr_*$ that yields a similar description as in the classical case, i.e.
\[ A_{aut}(\mathcal G^{\sqcup N}) = A_{aut}(\mathcal G) \wr_* S_N^+, \]
if $\mathcal G$ is connected. Here, the classical permutation group $S_N$ is replaced by the \emph{quantum permutation group} $S_N^+$ which was introduced by Wang in \cite{Wa}.
The free wreath product of a CMQG $\mathbb G = (A, (u_{ij}))$ by the quantum permutation group $S_N^+ = (C(S_N^+), \Delta_{S_N^+})$ is a quotient of the free product $A^{*N} * C(S_N^+)$ and is therefore neither a free product nor a tensor product.
It is a fundamental result by Woronowicz \cite{Wor87} that a CMQG always comes with a unique invariant state, called the \emph{Haar state} which one can use to obtain a reduced version $C^*_r(\mathbb G)$ of a CMQG $\mathbb G$. In the preliminary section of this article, these concepts will be discussed in more detail. The main theorem of the article is the following:

\begin{thmx} \label{maintheorem}
Let $\mathbb G$ be a CMQG with tracial Haar state. Then, for arbitrary $N \in \mathbb N$, the Haar state of the free wreath product $\mathbb G \wr_* S_N^+$ is tracial as well. Moreover, if $N \geq 8$, the reduced $C^*$-algebra $ C^*_r(\mathbb G \wr_* S_N^+)$ is simple with unique trace and its envelopping von Neumann algebra $L^{\infty}(\mathbb G \wr_* S_N^+)$ is a $II_1$-factor which does not have property $\Gamma$.
\end{thmx}

The proof of Theorem \ref{maintheorem} is in large parts an adaption of Lemeux's ideas in \cite{Lem2} to this more general situation. It is also heavily based on results and ideas of Powers (\cite{Pow75}), Banica (\cite{Ban97}) and Brannan (\cite{Bra2}).
\begin{ack}
This article is part of the authors master's thesis done under the supervision of Roland Speicher and Moritz Weber. The author wishes to thank them for their patience and helpful suggestions. Further thanks go to Fran\c{c}ois Lemeux and Pierre Tarrago for useful conversations on the subject.
\end{ack}

\section{Preliminaries}
\subsection{Compact Matrix Quantum Groups} We will summarise the basic facts on CMQGs that we will need throughout this paper. For a more detailed introduction to the subject the author recommends the excellent book \cite{tim}.

\begin{definition}[\cite{Wor87}]
A \emph{compact matrix quantum group} $\mathbb G = (A, u)$ is a unital $C^*$-algebra $A$ together with a unitary matrix $u = (u_{ij})_{1 \leq i,j \leq N} \in M_N(A), \ N \geq 1$, such that
\begin{enumerate}[(1)]
\item the elements $u_{ij}, \ 1 \leq i,j \leq N, $ generate $A$ as a $C^*$-algebra,
\item the conjugate $\bar{u}$ of $u$ is invertible,
\item there is a $*$-homomorphism $\Delta: A \to A \otimes A$ with 
$ \Delta(u_{ij}) = \sum_{k=1}^N u_{ik} \otimes u_{kj}$ for all $1 \leq i,j \leq N. $
\end{enumerate}

The matrix $u$ is called the \emph{fundamental corepresentation} of $\mathbb G$ and the $*$-homomorphism $\Delta$ is called the \emph{comultiplication}. We will often denote the $C^*$-algebra $A$ by $A = C(\mathbb G)$.
\end{definition}
A CMQG $\mathbb G = (A, u)$ always admits a \emph{Haar state} $h$ which is a state $h: A \to \mathbb C$ that is uniquely determined by the invariance condition
\[ (h \otimes \id_A) \Delta(a) = (\id_A \otimes h) \Delta(a) = h(a) 1_A \ \ (a \in A). \]
We denote the Hilbert space obtained from the GNS-construction with respect to $h$ by $L^2(\mathbb G)$  and the corresponding GNS-representation by $\pi_h: C(\mathbb G) \to B(L^2(\mathbb G))$. The pair $\mathbb G_r = (C^*_r(\mathbb G), \pi_h(u))$, where $C^*_r(\mathbb G) = \pi_h(A) \subset B(L^2(\mathbb G)) $ and $\pi_h(u)_{ij} = \pi_h(u_{ij})$, is a CMQG as well and its comultiplication $\Delta_r: C^*_r(\mathbb G) \to C^*_r(\mathbb G) \otimes C^*_r(\mathbb G)$ is related to the comultiplication $\Delta$ of $\mathbb G$ by $\Delta = \Delta_r \circ \pi_h$. $\mathbb G_r$ is called the \emph{reduced version} of $\mathbb G$ and its Haar state $h_r$ is given by $h = h_r \circ \pi_h$. Moreover, the envelopping von Neumann algebra of $C^*_r(\mathbb G)$ is denoted by $L^{\infty}(\mathbb G)$. An important example of a CMQG is the \emph{quantum permutation group} $S_N^+$. From the viewpoint of free probability theory, it is an appropriate free version of the usual permutation group $S_N$ in the sense that there is a de Finetti-type characterisation of free independence in terms of "permutation" by $S_N^+$ (see \cite{KoSp}). It is defined in the following way:

\begin{definition} \label{quantumperm}
Let $N \in \mathbb N$ and let $C(S_N^+)$ be the universal unital $C^*$-algebra with generators $u_{ij}, \ 1 \leq i,j \leq N$ satisfying the following relations:
\begin{enumerate}[(1)]
\item $u_{ij}$ is a projection for all $1 \leq i,j \leq N$,
\item for every $1 \leq i \leq N$ the projections $u_{i1},\dots,u_{iN}$ are orthogonal and $\sum_{j=1}^N u_{ij} = 1$, 
\item for every $1 \leq j \leq N$ the projections $u_{1j},\dots,u_{Nj}$ are orthogonal and $\sum_{i=1}^N u_{ij} = 1$.
\end{enumerate}
The pair $S_N^+ =(A_s(N), (u_{ij}))$ is a CMQG called the \emph{quantum permutation group}.
\end{definition}

\subsection{Free wreath products}
In this paper, we are mainly interested in CMQGs arising as a free wreath product of a CMQG $\mathbb G$ by the quantum permutation group $S_N^+$. The free wreath product construction was introduced by Bichon in \cite{Bic04} as a quantum analogue of the classical wreath product of groups. For $1 \leq k \leq N$, we denote the $k$-th canonical embedding of $C( \mathbb G)$ into the free product $C( \mathbb G)^{*N} * C(S_N^+)$ by $\nu_k$.

\begin{definition}[\cite{Bic04}, Definition 3.1] \label{wreathdef}
Let $\mathbb G = (C( \mathbb G), (v_{kl})) $ be a CMQG, let $ N \geq 1$ be an integer and consider the quantum permutation group $S_N^+ = (C(S_N^+), (u_{ij})_{1 \leq i,j \leq N})$. By $C(\mathbb G \wr_* S_N^+)$ we denote the quotient of $C( \mathbb G)^{*N} * C(S_N^+)$ by the closed two-sided ideal generated by the elements
\[ \nu_k(a) u_{ki} - u_{ki} \nu_k(a), \ \ \  1 \leq i,k \leq N, \ \ a \in C( \mathbb G). \]
 The \emph{free wreath product} $\mathbb G \wr_* S_N^+ $ is the CMQG defined by $\mathbb G \wr_* S_N^+ = (C(\mathbb G \wr_* S_N^+), (\nu_i(v_{kl}) u_{ij}) )$.
\end{definition}

\begin{rem}
\begin{enumerate}[(1)]
\item Note that, in Bichons original work \cite{Bic04}, one considers a compact quantum group $\mathbb G$ instead of a compact \emph{matrix} quantum group. However, it is easy to see that the free wreath product respects the structure of a CMQG as well, i.e. that $\mathbb G \wr_* S_N^+ $ is indeed a CMQG whenever $\mathbb G$ is.
\item Let $\Delta_{\mathbb G}, \Delta_{S_N^+}$ denote the comultiplication on $\mathbb G, S_N^+$ respectively. Then, the comultiplication $\Delta$ on $\mathbb G \wr_* S_N^+ $ is given by $ \Delta(u_{ij}) = \sum_{k=1}^N u_{ik} \otimes u_{kj} $
for $1 \leq i,j \leq N$ and
\[ \Delta(\nu_i(a)) = \sum_{k=1}^N \nu_i \otimes \nu_k (\Delta_{\mathbb G}(a)) (u_{ik} \otimes 1), \ \ \ 1 \leq k \leq N, \ \ a \in C( \mathbb G). \] 
\end{enumerate}
\end{rem}

\subsection{Corepresentation theory}
 A ($n$-dimensional) \emph{corepresentation} of a CMQG $ \mathbb G$ is a matrix $u = (u_{ij})_{1 \leq i,j \leq n} \in M_{n}(C(\mathbb G))$ such that
$ \Delta(u_{ij}) = \sum_{k=1}^n u_{ik} \otimes u_{kj} $ for all $1 \leq i,j \leq n. $
We say that $u$ is a \emph{unitary corepresentation}, if in addition $u$ is a unitary element of $M_{n}(C(\mathbb G))$. Note that, whenever $u = (u_{ij})_{1 \leq i,j \leq n}$ is a corepresentation, the \emph{conjugate} $\bar{u} = (u_{ij}^*)_{1 \leq i,j \leq n}$ is a corepresentation as well. However, in general $\bar{u}$ may not be unitary, even if $u$ is.

\begin{definition} \label{intertw}
Let $ \mathbb G$ be a CMQG and let $u = (u_{ij})_{1 \leq i,j \leq n} \in M_{n}(C(\mathbb G))$ and $v = (v_{kl})_{1 \leq k,l \leq m}\in M_{m}(C(\mathbb G))$ denote two corepresentations of $\mathbb G$.
\begin{enumerate}[(1)]
\item The vector space
\[ \Hom(u,v) = \{ T \in B(\mathbb C^n, \mathbb C^m) \ ; \ (T \otimes 1_{C(\mathbb G)}) u = v (T \otimes 1_{C(\mathbb G)})  \} \]
is called the \emph{intertwiner space} from $u$ to $v$.
\item The corepresentation $u$ is called \emph{irreducible} if $\Hom(u,u) = \mathbb C \id$. Note that $u$ is irreducible if and only if $\bar{u}$ is.
\item The corepresentations $u$ and $v$ are called \emph{equivalent} if there exists an invertible intertwiner in $\Hom(u,v)$ and \emph{unitarily equivalent} if there exists a unitary intertwiner in $\Hom(u,v)$.
\item The corepresentation $u \otimes v = (u_{ij} v_{kl})_{ 1 \leq i,j \leq n \atop 1 \leq k,l \leq m} \in M_{nm}(\mathbb C) \otimes C(\mathbb G) $ is called the \emph{tensor product} of $u$ and $v$. If $u$ and $v$ are unitaries, so is $u \otimes v$.
\end{enumerate}
\end{definition}

It is also possible to define the notion of an infinite dimensional corepresentation of a compact quantum group $\mathbb G$ . However, it is a celebrated result by Woronowicz (see \cite{Wor95}) that every irreducible corepresentation of a compact quantum group is finite dimensional and equivalent to a unitary one. Moreover, every unitary corepresentation is unitarily equivalent to a direct sum of irreducibles.\\
We denote the set of equivalence classes of irreducible unitary corepresentations of a CMQG $\mathbb G$ by $Irr(\mathbb G)$ and
we fix a maximal family $\{ u^{\alpha} = (u_{ij}^{\alpha})_{1 \leq i,j \leq d_{\alpha}} \ ; \ \alpha \in Irr(\mathbb G) \}$ of irreducible unitary pairwise non-equivalent corepresentations, with $u^{\bar{\alpha}}$ denoting the representative of the equivalence class of $\overline{u^{\alpha}}$. The span of the coefficients of this maximal family is the unique norm dense $*$-Hopf subalgebra of $C(\mathbb G)$ and is denoted by $\Pol(\mathbb G)$. The comultiplication on $\Pol(\mathbb G)$ is the restriction of the comultiplication on $C(\mathbb G)$ and the coinverse $\kappa: \Pol(\mathbb G) \to \Pol(\mathbb G)$ is the antihomomorphism given by $ \kappa(u_{ij}^{\alpha}) = (u_{ji}^{\alpha})^*, \ ( 1 \leq i,j \leq d_{\alpha}, \ \alpha \in I)$. The counit $\epsilon: \Pol(\mathbb G) \to \mathbb C$ is the $*$-character given by $\epsilon(u_{ij}^{\alpha}) = \delta_{ij}, \ ( 1 \leq i,j \leq d_{\alpha}, \ \alpha \in I).$
In addition, the restriction of the Haar state $h$ to $\Pol(\mathbb G)$ is faithful,
and the set $ \{ u_{ij}^{\alpha} \ ; \ 1 \leq i,j \leq d_{\alpha}, \ \alpha \in I \}$ is an orthogonal basis of the GNS-Hilbert space $L^2(\mathbb G)$. \\
The results in \cite{LemTa} have been obtained under the assumption that the Haar states of the CMQGs involved are tracial. A CMQG for which this holds is said to be of \emph{Kac type} and it is a well known result of Baaj and Skandalis (see \cite{BaSka}) that the Haar state on $\mathbb G$ is a trace if and only if the coinverse $\kappa$ extends continuously to a $*$-antihomomorphism on $C(\mathbb G)$ which again is equivalent to $\kappa^2 = \id$. Whenever we are in the Kac type setting, the conjugate $\bar{u}$ of an irreducible unitary corepresentation $u$ is unitary as well. Hence, we can always assume $ u^{\bar{\alpha}} = \overline{u^{\alpha}}$ in the above notation. \\
Note that, one can also define free wreath products on the level of $(*)$-Hopf algebras (see \cite[Definition 2.2]{Bic04}). In particular, if $\mathbb G$ is a CMQG and $\Pol(\mathbb G)$ is its unique dense $*$-Hopf algebra, we have $\Pol( \mathbb G \wr_* S_N^+) \cong \Pol(\mathbb G) \wr_* \Pol(S_N^+)$ as $*$-Hopf algebras. \\
Since the Haar state $h$ of a CMQG $\mathbb G$ is faithful on the underlying Hopf algebra, we have $\Pol(\mathbb G) \cong \Pol(\mathbb G_r)$ and hence one can derive many interesting results on the reduced version $\mathbb G_r$ from an understanding of the corepresentation theory on the full level. In particular, it is useful to know how tensor products of irreducible unitary corepresentations decompose into sums of irreducibles.\\
Let $M = \langle Irr(\mathbb G) \rangle$ denote the monoid formed by the words over $Irr(\mathbb G)$. We endow $M$ with the following operations:
\begin{enumerate}[(1)]
\item \emph{Involution}: $(\alpha_1, \dots, \alpha_k)^- = (\bar{\alpha}_k, \dots, \bar{\alpha}_1 )$,
\item \emph{concatenation}: for any two words $\alpha, \beta$ we set
\[ (\alpha_1, \dots, \alpha_k),(\beta_1, \dots, \beta_l) = (\alpha_1, \dots, \alpha_k, \beta_1, \dots, \beta_l ). \]
\end{enumerate}
\begin{thm}[\cite{LemTa}] \label{fusionwr}
Let $\mathbb G$ a CMQG of Kac type.
The equivalence classes of irreducible unitary corepresentations of $\mathbb G \wr_* S_N^+$ can be labelled by $\omega(x)$ with $x \in M$, with involution $\overline{\omega(x)} = \omega(\bar{x})$ and the fusion rules
\[ \omega(x) \otimes \omega(y) = \bigoplus_{ x = u,t \ ; \ y = \bar{t}, v} \omega(u,v) \oplus \bigoplus_{ x = u,t \ ; \ y = \bar{t}, v \atop u \neq \emptyset \ , \ v \neq \emptyset } \omega(u.v), \]
where $\omega(u.v)$ is defined as
\[ \omega((\alpha_1, \dots, \alpha_k).(\beta_1, \dots, \beta_l)) = \bigoplus_{ \gamma \subset \alpha_k \otimes \beta_1} \omega(\alpha_1, \dots, \gamma, \dots, \beta_l ) \]
for $u = (\alpha_1, \dots, \alpha_k)$ and $v = (\beta_1, \dots, \beta_l)$. This operation will be called \emph{fusion}.
Moreover, for all $\alpha \in Irr(\mathbb G)$ we have $r(\alpha) = \omega(\alpha) \oplus \delta_{\alpha, 1_{\mathbb G}} 1 $.
\end{thm}

\section{Simplicity and uniqueness of the trace}

In \cite{Lem2}, F. Lemeux proved that the reduced version of the free wreath product $H_N^+(\Gamma) = \hat{\Gamma} \wr_* S_N^+$ is simple and has a unique tracial state, namely the Haar state, for a discrete group $\Gamma$ with $| \Gamma | \geq 4$ and $N \geq 8$. The goal of this section is to generalise this result by showing that $\mathbb G \wr_* S_N^+$ is simple with unique trace whenever $\mathbb G$ is a CMQG of Kac type and $N \geq 8$. We will closely follow Lemeux's proof.
First, we observe that the free wreath product $\mathbb G \wr_* S_N^+$ inherits the traciality of its Haar state from its left component $\mathbb G$:

\begin{prop} \label{Kacwr}
Let $\mathbb G$ be a compact matrix quantum group of Kac type. Then $\mathbb G \wr_* S_N^+$ is of Kac type for all $N \geq 1$.
\end{prop}

\begin{proof}
We recall that the coinverse $\kappa: \Pol(\mathbb G \wr_* S_N^+) \to \Pol(\mathbb G \wr_* S_N^+)$ is given by $\kappa(v_{ij}) = v_{ji}^*$ where $v= (v_{ij})$ is an irreducible unitary corepresentation of $\mathbb G \wr_* S_N^+$. It suffices to show that $\kappa^2 = \id$.
Let $(u_{ij})$ be the fundamental corepresentation of $S_N^+$ and let $\alpha = (a_{ij}) \neq 1_{\mathbb G}$ be an irreducible unitary corepresentation of $\mathbb G$. As $\mathbb G$ is of Kac type, the irreducible corepresentation $\bar{\alpha}$ is unitary as well and hence 
$(\nu_i (a_{kl}) u_{ij}), \ (\nu_i (a_{kl}^*) u_{ij})$ are irreducible unitary corepresentations of $\mathbb G \wr_* S_N^+$. Therefore, for every $ 1\leq i,j \leq N$ and every $ 1\leq k,l \leq \dim\alpha$, we have
\[ \kappa(\nu_i (a_{kl}) u_{ij}) = (\nu_j (a_{lk}) u_{ji})^* = (u_{ji}\nu_j (a_{lk}))^* = \nu_j (a_{lk}^*) u_{ji} \]
and
\[ \kappa(\nu_i (a_{kl})^* u_{ij}) = \nu_j (a_{lk}) u_{ji},  \]
as $u_{ji}$ and $\nu_j (a_{lk})$ commute. Now it follows that
\[ \kappa(\nu_i (a_{kl})) = \kappa\bigg(\sum_{j=1}^N \nu_i (a_{kl}) u_{ij}\bigg) = \sum_{j=1}^N \nu_j (a_{lk}^*) u_{ji} \]
and therefore
\[ \kappa^2(\nu_i (a_{kl})) = \sum_{j=1}^N \kappa (\nu_j (a_{lk}^*) u_{ji}) = \sum_{j=1}^N \nu_i (a_{kl}) u_{ij} = \nu_i (a_{kl}). \]
For the trivial corepresentation $1_{\mathbb G}$ of $\mathbb G$, $\kappa^2(\nu_i(1_{\mathbb G})) = 1_{\mathbb G \wr_* S_N^+} = \nu_i(1_{\mathbb G})$ holds trivially for every $1 \leq i \leq N$. As  $\Pol(\mathbb G \wr_* S_N^+) \cong \Pol(\mathbb G) \wr_* \Pol(S_N^+)$ as $*$-Hopf algebras and as $\Pol(\mathbb G)$ is the linear span of the coefficients of the irreducible unitary corepresentations of $\mathbb G$, we have $\kappa^2 = \id$ on $\nu_i(\Pol(\mathbb G)) \subset \Pol(\mathbb G \wr_* S_N^+) $ for all $1 \leq i \leq N$. Moreover, since the embedding of $S_N^+$ into $\mathbb G \wr_* S_N^+$ is an injective morphism of quantum groups, we also have $\kappa^2 = \id$ on $\Pol(S_N^+)$. Hence $\kappa^2 = \id$ on $\Pol(\mathbb G \wr_* S_N^+)$.
\end{proof}

We consider the monoid $M = \langle Irr(\mathbb G) \rangle$ and we denote the empty word in $M$ by $\emptyset$. If $x = (\alpha_1,\dots,\alpha_k) \in M$, we will denote by $|x| = k$ the length of $x$ and if $A,B \subset M$, we set
\[ A \circ B = \{ z \in  M \ : \ \exists (x,y) \in A \times B \text{ such that } \omega(z) \subset \omega(x) \otimes \omega(y) \}. \]
Here, $\omega(z) \subset \omega(x) \otimes \omega(y)$ means that $\omega(z)$ appears as a direct summand in the decomposition of $\omega(x) \otimes \omega(y)$ into irreducibles. If $A \subset M$, we set $\bar{A}$ the set of conjugates $\bar{x}$ of elements $x \in A$. Also, we denote by $(\alpha,\dots)$ an element starting by $\alpha \in Irr(\mathbb G)$ and by $(\dots,\alpha)$ an element ending by $\alpha$. We need to partition $M$ into nice subsets.

\begin{notation} \label{lotnot}
We consider the trivial corepresentation $1_{\mathbb G} \in M$ and put:
\begin{itemize}
\item $1_{\mathbb G}^k$ the word $(1_{\mathbb G},\dots,1_{\mathbb G}) \in M$ of length $k$ with the convention $1_{\mathbb G}^0 = \emptyset$,

\item $E_1 = \bigcup \{ (1_{\mathbb G}, \dots) \} \cup \{ \emptyset \} $ the subset of words starting with $1_{\mathbb G}$,
\item $E_2= \bigcup_{k \in \mathbb N} \{1_{\mathbb G}^k \}$ the subset of words with only $1_{\mathbb G}$ as a letter,
\item $E_3 = E_1 \backslash E_2$,
\item $S = M \backslash E_2$,
\item $G_1  = \bigcup_{\alpha \neq 1_{\mathbb G} \atop \alpha \in Irr(\mathbb G)} \{ (\alpha, \dots) \}$ the subset of words starting with any $\alpha \neq 1_{\mathbb G}$,
\item $G_2 = \bigcup_{\alpha, \alpha' \neq 1_{\mathbb G}} \{ (\alpha, \dots, \alpha') \} $ the subset of words starting with any $\alpha \neq 1_{\mathbb G}$ and ending with any $\alpha' \neq 1_{\mathbb G}$.
\end{itemize}
\end{notation}

We will later on have a closer look at the combinatorics of the sets defined above.
To transfer the combinatorial structure of those sets to the reduced $C^*$-algebra of $\mathbb G \wr_* S_N^+$, we will also need to define corresponding $*$-subalgebras of $ \Pol(\mathbb G \wr_* S_N^+)$.

\begin{notation} \label{morenot}
\begin{enumerate}[(1)]
\item By $\mathcal E \subset \Pol(\mathbb G \wr_* S_N^+)$ we denote the sub-$*$-algebra generated by the coefficients of $\omega(x), \ x \in E_2$ 

\item and by $\mathcal S \subset \Pol(\mathbb G \wr_* S_N^+)$ we denote the sub-$*$-algebra generated by the coefficients of $\omega(x), \ x \in S$.
\end{enumerate}
\end{notation}

Now, by \cite[Lemma 2.1, Proposition 2.2]{Ver04}, there exists a unique conditional expectation
$ P: C_r^*(\mathbb G \wr_* S_N^+) \twoheadrightarrow \bar{\mathcal E}^{ \| \cdot \|_r }$ such that the Haar state $h_{\bar{\mathcal E}^{ \| \cdot \|_r }}$ on $\bar{\mathcal E}^{ \| \cdot \|_r }$ and the Haar state $h$ on $C_r^*(\mathbb G \wr_* S_N^+) $ satisfy $h = h_{\bar{\mathcal E}^{ \| \cdot \|_r }} \circ P$ and $\ker(P) = \bar{\mathcal S}^{ \| \cdot \|_r }$.\\
We note that $P$ is realised by the compression of the projection $p$ onto the closure of $\mathcal E$ in $L^2(\mathbb G \wr_* S_N^+)$, i.e. $ Px = pxp \in p \bar{\mathcal E}^{ \| \cdot \|_r } p \cong \bar{\mathcal E}^{ \| \cdot \|_r }$
for all $x \in C_r^*(\mathbb G \wr_* S_N^+)$. Moreover, we have the decomposition
\[ C_r^*(\mathbb G \wr_* S_N^+) = \overline{(\mathcal E \oplus \mathcal S)}^{\| \cdot \|_r} = \bar{\mathcal E}^{ \| \cdot \|_r } \oplus \bar{\mathcal S}^{ \| \cdot \|_r } \]
as $S \sqcup E_2 = M$. Here, the symbol $\oplus$ denotes the direct sum of vector spaces (not the direct sum of $C^*$-algebras).
Our next step is to identify $C_r^*(S_N^+)$ as a sub-$C^*$-algebra of $C_r^*(\mathbb G \wr_* S_N^+)$ in terms of words in $M$.

\begin{prop} \label{iso}
$C_r^*(S_N^+)$ and $\bar{\mathcal E}^{ \| \cdot \|_r } \subset C_r^*(\mathbb G \wr_* S_N^+)$ are isomorphic as compact matrix quantum groups. 
\end{prop}

\begin{proof}
We will first show that $\Pol(S_N^+) \subset C(S_N^+)$ is isomorphic as a Hopf algebra to $\mathcal E \subset C(\mathbb G \wr_* S_N^+)$. To do so, we notice that by definition of $\mathbb G \wr_* S_N^+$ the natural embedding of $C(S_N^+)$ into $C(\mathbb G \wr_* S_N^+)$ is an isomorphism of CMQGs onto its range and hence we can consider $C(S_N^+)$ a $C^*$-subalgebra of $C(\mathbb G \wr_* S_N^+)$. By Theorem \ref{fusionwr}, the fundamental corepresentation $(u_{ij})$ of $S_N^+$ is given by 
\[(u_{ij}) = r(1_{\mathbb G}) = \omega(1_{\mathbb G}) \oplus 1 =  \omega(1_{\mathbb G}) \oplus \omega(\emptyset), \]
and it follows that $ \Pol(S_N^+) $ is the $*$-subalgebra of $\Pol(\mathbb G \wr_* S_N^+)$ generated by the coefficients of $\omega(1_{\mathbb G})$ and $\omega(\emptyset)$. As $1_{\mathbb G}, \emptyset \in E_2$, we get $\Pol(S_N^+) \subset \mathcal E$.\\ 
To prove the inclusion $ \mathcal E \subset \Pol(S_N^+)$, we need to show that the coefficents of $\omega(1_{\mathbb G}^k)$ lie in $\Pol(S_N^+)$ for every $k \in \mathbb N$. We do so by induction: \\
For $k=0,1$, we have already noticed that this is true. For $k \geq 2$ we assume that the assertion holds true for all $n < k$. The fusion rules in Theorem \ref{fusionwr} imply
\[ \omega (1_{\mathbb G}^{k-1}) \otimes \omega(1_{\mathbb G}) =  \omega (1_{\mathbb G}^k) \oplus \omega (1_{\mathbb G}^{k-1}) \oplus \omega (1_{\mathbb G}^{k-2}), \]
and hence the coefficients of $\omega (1_{\mathbb G}^k)$ can be written as linear combinations
of coefficients of $\omega (1_{\mathbb G}^{k-1}) \otimes \omega(1_{\mathbb G}), \omega (1_{\mathbb G}^{k-1})  $ and $\omega (1_{\mathbb G}^{k-2})$. This proves $ \mathcal E \subset \Pol(S_N^+)$ and hence $ \mathcal E = \Pol(S_N^+)$ and $ C(S_N^+) \cong \tilde{\mathcal E}$. In particular, their reduced versions are isomorphic, i.e.
 \[ C^*_r(S_N^+) \ \cong \ C^*_r(\tilde{\mathcal E}) \ \cong \ \overline{\mathcal E}^{ \| \cdot \|_r} \subset C_r^*(\mathbb G \wr_* S_N^+). \]
\end{proof}

Since $C^*_r(S_N^+)$ is simple for $N \geq 8$ by \cite{Bra2}, we obtain:
\begin{cor} \label{simplecor}
$\bar{\mathcal E}^{ \| \cdot \|_r } \subset C_r^*(\mathbb G \wr_* S_N^+)$ is simple with unique trace for $N \geq 8$.
\end{cor}

The next result shows that the subsets of the monoid $M$ defined in Notation \ref{lotnot} have certain stability properties:

\begin{lem} \label{propstab}
Let $ \mathbb G$ be a CMQG with $| Irr(\mathbb G) | \geq 2$, i.e. $\mathbb G \neq \mathbb C$, and let $\alpha \in Irr(\mathbb G)$ such that $ \alpha \neq 1_{\mathbb G} $. Moreover, let $G \subset S$ be a finite set and put $x_1 = (\alpha, 1_{\mathbb G}), x_2 = (\alpha, 1_{\mathbb G}^3), x_3 = (\alpha, 1_{\mathbb G}^5) \in M $. Then:
\begin{enumerate}[(1)]
\item $ S = E_3 \sqcup G_1$,
\item $(G_2 \circ E_1) \cap E_1 = \emptyset$,
\item $ (\{ x_t \} \circ G_1 )\cap (\{ x_s \} \circ G_1) = \emptyset$ if $t \neq s$,
\item $ \bigcup_{t=1}^3 \{ x_t \} \circ G_2 \circ \{ \bar{x}_t \} \subset G_2 $,
\item there is $x \in S$ such that $\{ x \} \circ G \circ \{ \bar{x} \} \subset G_2$.
\end{enumerate}
\end{lem}

\begin{proof}
\begin{enumerate}[(1)]
\item This assertion is obvious.
\item Let $(\alpha,\dots, \alpha' ) \in G_2$, i.e. $\alpha \neq 1_{\mathbb G} \neq \alpha'$. Then $\omega(\alpha,\dots, \alpha' ) \otimes \omega(\emptyset) = \omega(\alpha,\dots, \alpha' ).$
Also, for $(1_{\mathbb G},\dots) \in E_1$ and for every $t \in M$ such that $(\alpha,\dots, \alpha' ) = (u,t)$ and $(1_{\mathbb G},\dots) = (\bar{t},v)$, we get $u \neq \emptyset$ as $\alpha' \neq 1_{\mathbb G}$. Hence $(u,v)$ and $(u.v)$ start in $\alpha \neq 1_{\mathbb G}$ and by the fusion rules of $\mathbb G \wr_* S_N^+$ (Theorem \ref{fusionwr}), we have $(G_2 \circ E_1) \cap E_1 = \emptyset$.
\item Let $(\beta,\dots), (\gamma,\dots),(\delta,\dots)$ be words in $G_1$ such that $ \beta, \gamma, \delta \neq 1_{\mathbb G}$ are non-trivial irreducible corepresentations of $ \mathbb G$. Using Theorem \ref{fusionwr} once more, we obtain
\begin{align*}
\omega(x_1) \otimes \omega(\beta,\dots) &= \omega(\alpha, 1_{\mathbb G}, \beta, \dots ) \oplus \omega(\alpha, \beta, \dots), \\
\omega(x_2) \otimes \omega(\gamma,\dots) &= \omega(\alpha, 1_{\mathbb G}^3, \gamma, \dots ) \oplus \omega(\alpha, 1_{\mathbb G}^2, \gamma, \dots),\\
\omega(x_3) \otimes \omega(\delta,\dots) &= \omega(\alpha, 1_{\mathbb G}^5, \delta, \dots ) \oplus \omega(\alpha,1_{\mathbb G}^4, \delta, \dots).
\end{align*}
Thus, if $s \neq t$, any direct summand appearing in the tensor product of $\omega(x_s)$ and a corepresentation indexed by a word in $G_1$ does not appear as a direct summand of $\omega(x_t)$ and a corepresentation indexed by a word in $G_1$. Hence $(\{ x_t \} \circ G_1 )\cap (\{ x_s \} \circ G_1) = \emptyset,$ whenever $t \neq s$.
\item Now we consider an element $ (\beta, \dots, \beta') \in G_2$ where $\beta, \beta' \neq 1_{\mathbb G}$. Then we get
\begin{align*}
\omega(x_1) \otimes \omega(\beta, \dots, \beta') \otimes \omega(\bar{x}_1) &= \omega(\alpha,1_{\mathbb G} ) \otimes \omega(\beta, \dots, \beta') \otimes \omega(1_{\mathbb G}, \bar{\alpha}) \\
&=  \omega(\alpha,1_{\mathbb G},\beta, \dots, \beta',1_{\mathbb G}, \bar{\alpha}) \oplus \omega(\alpha, 1_{\mathbb G}, \beta, \dots, \beta', \bar{\alpha})\\
& \ \ \ \ \ \oplus \omega(\alpha, \beta, \dots, \beta',1_{\mathbb G}, \bar{\alpha}) \oplus \omega(\alpha, \beta, \dots, \beta', \bar{\alpha}). 
\end{align*}
As the words indexing the direct summands appearing on the righthand side neither start nor end in $1_{\mathbb G}$, we obtain $\{ x_1 \} \circ G_2 \circ \{ \bar{x}_1 \} \subset G_2,$ and a similar computation for $x_2$ and $x_3$ proves $\bigcup_{t=1}^3 \{ x_t \} \circ G_2 \circ \{ \bar{x}_t \} \subset G_2.$
\item Consider $x = (\alpha, 1_{\mathbb G}^k)$, where $k = \max \{|y|, \ y \in G \} +1$. For $y \in G \subset S$ we can write $y = (1_{\mathbb G}^{l-1}, h_l, \dots, h_{l'}, 1_{\mathbb G}^{m-l'})$, where $m > 1, \ 1 \leq l \leq l' \leq m$ and $h_l, h_{l'} \neq 1_{\mathbb G}$. Again, by using the fusion rules of Theorem \ref{fusionwr} we get
\begin{align*}
\omega(x) \otimes \omega(y) 
&= \bigoplus_{t=0}^{2(l-1)} \omega(\alpha, 1_{\mathbb G}^{k+l-1-t}, h_l, \dots, h_{l'}, 1_{\mathbb G}^{m-l'} )
\end{align*}
and hence
\begin{align*}
\omega(x) \otimes \omega(y) \otimes \omega(\bar{x}) =  \bigoplus_{t=0}^{2(l-1)} \bigoplus_{s=0}^{2(m-l')} \omega(\alpha, 1_{\mathbb G}^{k+l-1-t}, h_l, \dots, h_{l'}, 1_{\mathbb G}^{k+m-l' -s}, \bar{\alpha} ).
\end{align*}
As all directs summands appearing in this decomposition are indexed by words in $G_2$, it follows that $ \{ x \} \circ G \circ \{ \bar{x} \} \subset G_2. $
\end{enumerate}
\end{proof}

As in \cite{Lem2}, we will adapt the "modified Powers method" of T. Banica in \cite{Ban97}, where the simplicity of $C^*_r(U_N^+)$ is proven. The support $\supp(z)$ of an element $z \in \Pol(\mathbb G \wr_* S_N^+)$ is the smallest subset of $M$, such that $z$ can be written as a linear combination of coefficients of elements $\omega(x), \ x \in \supp(z)$. Banica's crucial result for our proof is the following:

\begin{prop}[\cite{Ban97}, Proposition 8] \label{BanProp8}
Let $\mathbb G$ be a CMQG of Kac type and let $Irr(\mathbb G) = C \sqcup D$ be a partition of $Irr(\mathbb G)$ into non-empty sets $C, D$. Moreover, let $y_1,y_2,y_3 \in Irr(\mathbb G)$ such that $(y_t \circ D) \cap (y_s \circ D) = \emptyset$, if $t \neq s$. Then there is a unital linear map $T: C_r^*(\mathbb G) \to C_r^*(\mathbb G) $ with the following properties:
\begin{enumerate}[(1)]
\item There is a finite family $(a_i)$ in $\Pol(\mathbb G)$ such that $T(z) = \sum_i a_i z a_i^*$ for all $z \in C_r^*(\mathbb G)$.
\item $T$ is $\tau$-preserving for any trace $\tau \in C_r^*(\mathbb G)^*$.
\item For all self-adjoint $z \in \Pol(\mathbb G)$ with $ (\supp(z) \circ C) \cap C = \emptyset$, we have $ \| T(z) \|_r \leq 0.95 \|z \|_r $ and $ \supp(T(z)) \subset \bigcup_{i=1}^3 y_i \circ \supp(z) \circ \bar{y}_i. $
\end{enumerate}
\end{prop}

We are now ready to prove the first part of Theorem \ref{maintheorem}.

\begin{lem} \label{main1}
Let $\mathbb G$ be a CMQG of Kac type. Then, $C_r^*(\mathbb G \wr_* S_N^+)$ is simple for all $N \geq 8$.

\end{lem}

\begin{proof}
If $|Irr(\mathbb G)| =1$, we have $\mathbb G \wr_* S_N^+ = S_N^+$ whose reduced $C^*$-algebra is simple by \cite[Corollary 5.12]{Bra2}. Hence we may assume $|Irr(\mathbb G)| \geq 2$. We put $\mathcal E' = \overline{\mathcal E}^{ \| \cdot \|_r}$ and $\mathcal S' = \overline{\mathcal S}^{ \| \cdot \|_r}$, where $ \| \cdot \|_r$ denotes the norm on $C_r^*(\mathbb G \wr_* S_N^+)$. Let $J \triangleleft C_r^*(\mathbb G \wr_* S_N^+)$ be an ideal. We have to prove that either $J = \{ 0 \}$ or $J= C_r^*(\mathbb G \wr_* S_N^+)$. Recall that there is a unique conditional expectation $P: C_r^*(\mathbb G \wr_* S_N^+) \twoheadrightarrow \mathcal E'$. Hence it holds that $v P(x) w = P(vxw) \in P(J) \ \ (v,w \in \mathcal E', \ x \in J ).$
Moreover, as $P$ is realised by the compression by the projection $p \in B(L^2(\mathbb G \wr_* S_N^+))$ onto the $ \| \cdot \|_2$-closure of $\mathcal E'$, it follows that $P(J) = pJp \subset \mathcal E'$ is norm-closed and thus, $P(J)$ is a closed two-sided ideal. By the simplicity of $\mathcal E'$ (see Corollary \ref{simplecor}), we obtain $P(J) = \{ 0 \}$ or $P(J) = \mathcal E'$.
\begin{case}
Let $P(J) = \{ 0 \}$, i.e. $J \subset \ker P = \mathcal S'$. As $1_{\mathbb G \wr_* S_N^+} \notin \mathcal S'$, Schur's orthogonality relations (cf. \cite{tim}) imply $h(y) = 0$ for all $y \in \mathcal S'$, i.e. $ \mathcal S' \subset \ker h$. For $y \in J$, we get $y^* y \in J$ as $J$ is an ideal and therefore $h(y^*y) = 0$, since $J \subset \mathcal S' \subset \ker h$. But as $h$ is faithful on $C_r^*(\mathbb G \wr_* S_N^+)$, this means $y = 0$ and hence $J = \{ 0 \}$.
\end{case}
\begin{case}
Now, let $P(J) = \mathcal E'$. Since $1_{\mathbb G \wr_* S_N^+} \in \mathcal E'$, there is $y \in J$ such that $P(y) = 1_{\mathbb G \wr_* S_N^+}$ and hence we can write $y = P(y) - z = 1-z$ with $z \in \mathcal S'$. We choose $z_0 \in \mathcal S$ such that $\| z- z_0 \|_r < \frac{1}{2}$. By putting $G = \supp(z_0) \subset S$ in Lemma \ref{propstab}, we find $x \in S$ such that $\{x\} \circ \supp (z_0) \circ \{ \bar{x} \} \subset G_2. $ Let us denote the coefficients of the irreducible unitary corepresentation $\omega(x)$ by $\tilde{a}_{ij}, \ 1 \leq i,j \leq \dim \omega(x)$. Then, it follows that the element $z' = \sum_{i,j} a_{ij} z_0 a_{ij}^*$, 
where $a_{ij} = (\dim \omega(x))^{- \frac{1}{2}} \  \tilde{a}_{ij}$, fulfills 
\[ \supp(z') \subset \{x\} \circ \supp (z_0) \circ \{ \bar{x} \} \subset G_2, \]
and since $ G_2 = \bar{G}_2$, the same holds for the self-adjoint elements $Re (z') = \frac{1}{2} (z' + (z')^*)$ and $ Im (z') = \frac{1}{2i} (z' - (z')^*) $, i.e. $\supp (Re (z')), \ \supp (Im (z')) \subset G_2. $ \\
We also note that the mapping $T_0: w \mapsto \sum_{i,j} a_{ij} w a_{ij}^*$ is trace-preserving and completely positive and unital, since $\omega(x)$ is unitary. \\
We will now apply Proposition \ref{BanProp8} to $Re(z')$ and $Im(z')$. To do so, we note that the monoid $M$ indexing the irreducible corepresentations of $\mathbb G \wr_* S_N^+$ can be partitioned as $M = E_1 \sqcup G_1$ and by Lemma \ref{propstab}, the words $x_1 = (\alpha, 1_{\mathbb G}), x_2 = (\alpha, 1_{\mathbb G}^3), x_3 = (\alpha, 1_{\mathbb G}^5)$ satisfy
$  (\{ x_t \} \circ G_1 )\cap (\{ x_s \} \circ G_1) = \emptyset$, whenever $t \neq s$. Furthermore, by part (2) of Lemma \ref{propstab}, we have $(G_2 \circ E_1) \cap E_1 = \emptyset$ and hence
\begin{align*}
 (\supp(Re(z')) \circ E_1) \cap E_1 = \emptyset, \ \ (\supp(Im(z')) \circ E_1) \cap E_1 = \emptyset.
\end{align*}
Thus, by Proposition \ref{BanProp8} there is a unital completely positive trace-preserving map 
$ T_1: C_r^*(\mathbb G \wr_* S_N^+) \to C_r^*(\mathbb G \wr_* S_N^+)$, such that
\begin{itemize}
\item $T_1(w) = \sum_{i} c_i w c_i^*$ for some finite family $(c_i)$ in $\Pol(\mathbb G \wr_* S_N^+)$,
\item $\| T_1( Re(z')) \|_r \leq 0.95 \| Re(z') \|_r $, $\| T_1( Im(z')) \|_r \leq 0.95 \| Im(z') \|_r $,
\item $ \supp (T_1( Re(z'))), \supp (T_1( Im(z'))) \subset \bigcup_{t=1}^3 \{ x_t \} \circ G_2 \circ \{ \bar{x}_t \} \subset G_2 $.
\end{itemize}
Since $T_1( Re(z')), T_1( Im(z'))$ are again self-adjoint with 
\begin{align*}
 \supp (T_1( Re(z'))) \circ E_1) \cap E_1 = \emptyset, \ \
 \supp (T_1( Im(z'))) \circ E_1) \cap E_1 = \emptyset,
\end{align*}
we may apply Proposition \ref{BanProp8} iteratively in order to obtain a finite family $(d_i)$ in \\ $\Pol(\mathbb G \wr_* S_N^+)$ such that 
\[
\big\| \sum_{i} d_i Re(z') d_i^* \big\|_r < \frac{1}{4}, \ 
\big\| \sum_{i} d_i Im(z') d_i^* \big\|_r < \frac{1}{4},
\]
and hence $\big\| \sum_{i} d_i z' d_i^* \big\|_r < \frac{1}{2}.$
By plugging in $z' = \sum_{i,j} a_{ij} z_0 a_{ij}^*$, we get a finite family $(b_i)$ in $\Pol(\mathbb G \wr_* S_N^+)$ such that $\big\| \sum_{i} b_i z_0 b_i^* \big\|_r < \frac{1}{2}. $
We note that the mapping $T:w \mapsto \sum_{i} b_i w b_i^*$ is unital, completely positive and trace-preserving by construction and therefore by the Russo-Dye Theorem we obtain $\| T \| = \| T(1) \| = 1.$
Altogether, the invertibility of the element $\sum_{i} b_i y b_i^* \in J$ follows from the calculation
\begin{align*}
\big\| 1 - \sum_{i} b_i y b_i^* \big\|_r &= \big\| \sum_{i} b_i (1-y) b_i^* \big\|_r 
= \big\| \sum_{i} b_i z b_i^* \big\|_r \\
&\leq \big\| \sum_{i} b_i z_0 b_i^* \big\|_r + \big\| \sum_{i} b_i (z-z_0) b_i^* \big\|_r \\
&\leq \big\| \sum_{i} b_i z_0 b_i^* \big\|_r + \| T \| \| z- z_0 \|_r 
< 1.
\end{align*} 
\end{case}
Hence $J = C_r^*(\mathbb G \wr_* S_N^+)$.
\end{proof}

The methods of the last proof also yield:

\begin{lem} \label{main2}
Let $\mathbb G$ be a CMQG of Kac type. Then, $C_r^*(\mathbb G \wr_* S_N^+)$ has a unique tracial state $h$ for all $N \geq 8$. In particular, $L^{\infty}(\mathbb G \wr_* S_N^+)$ is a $II_1$-factor for all $N \geq 8$.
\end{lem}

\begin{proof}
We will show that any trace state $\tau$ on $C_r^*(\mathbb G \wr_* S_N^+)$ coincides with the Haar state $h$. To do so, let $z = z^* \in \mathcal S$. On the one hand we have $h(z) = 0$ by Schur's orthogonality relations and on the other hand, for all $\epsilon > 0$ we can repeat the method of the proof of Theorem \ref{main1}
to find a finite family $(b_i)$ in $\Pol(\mathbb G \wr_* S_N^+)$ such that $ \big\| \sum_{i} b_i z b_i^* \big\|_r < \epsilon.$
Furthermore, the mapping $T:w \mapsto \sum_{i} b_i w b_i^*$ is unital, completely positive and trace-preserving which implies that $\tau(z) = \tau(T(z)) < \epsilon.$
As this holds for all $\epsilon > 0$, we have $\tau(z) = 0$. Since every element in $\mathcal S$ can be written as a linear combination of two self-adjoint elements, $\tau$ and $h$ coincide on $\mathcal S$. \\
For $z \in \mathcal E$, we have $\tau(z) = h(z)$ by the uniqueness of the trace on $\mathcal E' \cong C_r^*(S_N^+)$. Hence, $\tau$ and $h$ coincide on $\Pol(\mathbb G \wr_* S_N^+)$ and by continuity also on $C_r^*(\mathbb G \wr_* S_N^+)$.
\end{proof}

\begin{rem}
Up to this point, we have only considered the case where $N \geq 8$. Although the cases $4 \leq N \leq 7$ remain open, it is easy to see that factoriality will in general not hold if $1 \leq N \leq 3$. For example, we can simply consider $S_N^+ = S_N^+ \wr_* S_1^+ = S_1^+ \wr_* S_N^+ = S_N$ for $N=2,3$. Of course, since in this case $S_N^+ = S_N$ is commutative, its associated von Neumann algebra is not a factor. If $N=2$, the fact that $L^{\infty}(\mathbb G \wr_* S_2)$ is not a factor does not even depend on the choice of $\mathbb G$, as $L^{\infty}(S_2)$ is contained in the center of $L^{\infty}(\mathbb G \wr_* S_2)$.
\end{rem}

\section{Fullness of $L^{\infty}(\mathbb G \wr_* S_N^+)$ }

In this section, we will prove that the $II_1$-factor $L^{\infty}(\mathbb G \wr_* S_N^+)$ is full, i.e. it does not have property $ \Gamma$ whenever $N \geq 8$. We recall the following definition.

\begin{definition}
Let $(M, \tau)$ be a $II_1$-factor with unique faithful normal trace $\tau$. 
\begin{enumerate}[(1)]
\item A sequence $(x_n)$ in $M$ is said to be \emph{asymptotically central}, if $ \| x_n y - y x_n \|_{L^2(M)} \to 0$ for all  $y \in M .$
\item A sequence $(x_n)$ in $M$ is said to be \emph{asymptotically trivial}, if $\| x_n  - \tau(x_n) 1 \|_{L^2(M)} \to 0 .$
\item The $II_1$-factor $M$ is called \emph{full}, if every bounded asymptotically central sequence is asymptotically trivial. If $M$ is not full, we say that $M$ has \emph{property} $\Gamma$.
\end{enumerate}
\end{definition}

Our proof is a simple generalization of Lemeux's proof of the fullness of $L^{\infty}(H_N^+(\Gamma))$. We will denote the $L^2(\mathbb G \wr_* S_N^+)$-norm by $\| \cdot \|_2$. 

\begin{notation}
Let $M = \langle Irr (\mathbb G) \rangle$ denote the monoid indexing the irreducible corepresentations of $\mathbb G \wr_* S_N^+$. For a subset $B \subset M $ we denote
\[ L^2(B) = \overline{\spann \{ \Lambda_h(x) \ ; \ \supp(x) \subset B  \}}^{\| \cdot \|_2} \subset L^2(\mathbb G \wr_* S_N^+), \]
where $\Lambda_h$ is the GNS-map with respect to the Haar state $h$.
\end{notation}

We notice now that the GNS-Hilbert space of $\mathbb G \wr_* S_N^+$ decomposes as the orthogonal sum $L^2(\mathbb G \wr_* S_N^+) = L^2(E_2) \oplus L^2(S),$
where $E_2= \bigcup_{k \in \mathbb N} \{1_{\mathbb G}^k \}$ is the subset of words with only $1_{\mathbb G}$ as a letter and $S = M\backslash E_2$. In particular, we have
$ L^{\infty}(\mathbb G \wr_* S_N^+) = \overline{\mathcal E}^{wo} \oplus \overline{\mathcal S}^{wo}$, where $\overline{\mathcal E}^{wo}$  (respectively $\overline{\mathcal S}^{wo}$) denotes the closure of $\mathcal E$ (respectively $\mathcal S$) (c.f. Notation \ref{morenot}) in the weak operator topology.

\begin{prop} \label{Hilfsprop2}
Let $N \geq 8$ be an integer. If every bounded asymptotically central sequence in $ \mathcal S$ is trivial, then $L^{\infty}(\mathbb G \wr_* S_N^+)$ is full.
\end{prop}

\begin{proof}
By a straightforward density argument, it suffices to show that every bounded asymptotically central sequence $(x_n)_{n \in \mathbb N}$ in $\Pol(\mathbb G \wr_* S_N^+)$ is asymptotically trivial. For all $n \in \mathbb N$ we write $x_n = y_n + z_n$, where $y_n \in \mathcal E$ and $z_n \in \mathcal S$ and we denote the orthogonal projection onto $L^2(E_2)$ by $ P: L^2(\mathbb G \wr_* S_N^+) \to L^2(E_2).$ Recall that $P|_{L^{\infty}(\mathbb G \wr_* S_N^+)}: L^{\infty}(\mathbb G \wr_* S_N^+) \to \overline{\mathcal E}^{wo}$ is the conditional expectation on $\overline{\mathcal E}^{wo}$.
Since the restriction to $\overline{\mathcal E}^{wo}$ of the Haar state on $L^{\infty}(\mathbb G \wr_* S_N^+)$ is the Haar state on $\overline{\mathcal E}^{wo}$, we have
\begin{align*}
\| y_n a - a y_n \|_{L^2(\overline{\mathcal E}^{wo})} = \| y_n a - a y_n \|_{L^2(E_2)} 
= \| P(x_n) a - a P(x_n) \|_2 
\leq \|x_n a - a x_n \|_2
\to 0
\end{align*}
for all $a \in \overline{\mathcal E}^{wo} $. Hence, the sequence $(y_n)_{n \in \mathbb N}$ is asymptotically central in $\overline{\mathcal E}^{wo}$ and it is clear that the isomorphism of compact quantum groups $\overline{\mathcal E}^{\| \cdot \|_r} \cong C^*_r(S_N^+) $ in Proposition \ref{iso} extends to an isomorphism $\overline{\mathcal E}^{wo} \cong L^{\infty}(S_N^+)$. \\
By \cite{Bra2}, $L^{\infty}(S_N^+)$ is full and hence $(y_n)_{n \in \mathbb N}$ is asymptotically trivial. In particular, it is asymptotically central in $L^{\infty}(\mathbb G \wr_* S_N^+)$. This implies that the sequence $(z_n)_{n \in \mathbb N} = (x_n-y_n)_{n \in \mathbb N}$ is bounded asymptotically central in $L^{\infty}(\mathbb G \wr_* S_N^+)$ and by assumption it is asymptotically trivial. Hence $(x_n)_{n \in \mathbb N}$ is asymptotically trivial in $L^{\infty}(\mathbb G \wr_* S_N^+)$.
\end{proof}

The last proposition shows that we only need to deal with bounded asymptotically trivial sequences in $\mathcal S$. Recall that in Notation \ref{lotnot} we defined $G_1  = \bigcup_{\alpha \neq 1_{\mathbb G} \atop \alpha \in Irr(\mathbb G)} \{ (\alpha, \dots) \}$ as the subset of words starting with any $\alpha \neq 1_{\mathbb G}$ and $E_3 = E_1 \backslash E_2$ as the subset of words starting in $1_{\mathbb G}$ but containing a letter different than $1_{\mathbb G}$. Note that $ S = E_3 \sqcup G_1$.

\begin{lem} \label{hilfslem1}
Let $ 1_{\mathbb G} \neq \alpha \in Irr (\mathbb G)$. With the notation above, we have
\begin{enumerate}[(1)]
\item $\{ (\alpha) \} \circ E_3 \circ \{ (\bar{\alpha}) \} \subset G_1$,
\item $\{ 1_{\mathbb G}^i \} \circ G_1 \circ \{ 1_{\mathbb G}^i \} \subset E_3 \text{ for } i=2,4$,
\item $ \big( \{ 1_{\mathbb G}^2 \} \circ G_1 \circ \{ 1_{\mathbb G}^2 \} \big) \cap \big( \{ 1_{\mathbb G}^4 \} \circ G_1 \circ \{ 1_{\mathbb G}^4 \} \big) = \emptyset. $ 
\end{enumerate}
\end{lem}

\begin{proof}
\begin{enumerate}[(1)]
\item Let $t \in E_3$, i.e. $t = (1_{\mathbb G}, \beta_1,\dots,\beta_l)$ with $l \geq 1$ and $\beta_i \neq 1_{\mathbb G}$ for at least one $i \in \{1, \dots, l \}$. From Theorem \ref{fusionwr} it follows that
\begin{align*}
\omega(\alpha) \otimes \omega(t) \otimes \omega(\bar{\alpha})
&= \omega(\alpha, 1_{\mathbb G}, \beta_1,\dots,\beta_l, \bar{\alpha}) \oplus \delta_{\beta_l, \alpha} \  \omega(\alpha, 1_{\mathbb G}, \beta_1,\dots,\beta_{l-1}) \\
& \ \ \ \oplus \ \ \omega(\alpha, \beta_1,\dots,\beta_l, \bar{\alpha}) \oplus \delta_{\beta_l, \alpha} \ \omega(\alpha, \beta_1,\dots,\beta_{l-1}) \\
& \ \ \ \oplus \bigoplus_{ \gamma \subset \beta_l \otimes \bar{\alpha}} \omega(\alpha, 1_{\mathbb G}, \beta_1,\dots,\beta_{l-1}, \gamma) \\
& \ \ \ \oplus \bigoplus_{ \gamma \subset \beta_l \otimes \bar{\alpha}} \omega(\alpha, \beta_1,\dots,\beta_{l-1}, \gamma). 
\end{align*}
Since all of the words appearing in this direct sum start in $\alpha \neq 1_{\mathbb G}$, we obtain $\{ (\alpha) \} \circ E_3 \circ \{ (\bar{\alpha}) \} \subset G_1.$
\item Let $(\beta, \dots)$ be a word in $G_1$, i.e. $\beta \neq 1_{\mathbb G}$ and let $i \in \{ 2,4\}$. We have
\[ \omega(1^i_{\mathbb G}) \otimes \omega(\beta,\dots) = \omega(1^i_{\mathbb G},\beta,\dots) \oplus \omega(1^{i-1}_{\mathbb G},\beta,\dots). \]
Since $\beta \neq 1_{\mathbb G}$, the tensor product
$ \big(\omega(1^i_{\mathbb G},\beta,\dots) \oplus \omega(1^{i-1}_{\mathbb G},\beta,\dots) \big) \otimes \omega(1^i_{\mathbb G}) $
will only produce subcorepresentations of the form $\omega(1^i_{\mathbb G},\beta,\dots)$ and $\omega(1^{i-1}_{\mathbb G},\beta,\dots)$. This proves assertion (2).
\item This follows immediately from the above calculations since corepresentations appearing as direct summands of 
$ \omega(1^2_{\mathbb G}) \otimes \omega(\beta,\dots) \otimes \omega(1^2_{\mathbb G}) \ \ (\beta \neq 1_{\mathbb G}) $
are indexed by words starting in $1_{\mathbb G}$ or $1^2_{\mathbb G}$ and corepresentations appearing as direct summands of 
$ \omega(1^4_{\mathbb G}) \otimes \omega(\beta,\dots) \otimes \omega(1^4_{\mathbb G}) \ \ (\beta \neq 1_{\mathbb G}) $
are indexed by words starting in $1^3_{\mathbb G}$ or $1^4_{\mathbb G}$. 
\end{enumerate}
\end{proof}

Note that in the previous lemma we assume that $|Irr(\mathbb G)| \geq 2$. The case $|Irr(\mathbb G)| = 1$ corresponds to $S_N^+$ for which we already have the desired fullness result. \\
Recall that, since the coefficients of the irreducible unitary corepresentations of $\mathbb G \wr_* S_N^+$ form an orthogonal basis of $L^2(\mathbb G \wr_* S_N^+)$ and since $ S = E_3 \sqcup G_1$, the Hilbert spaces $H_1 := L^2(E_3)$ and $H_2 := L^2(G_1)$ are orthogonal subspaces of $L^2(\mathbb G \wr_* S_N^+)$. We put $H := H_1 \oplus H_2$ and  
\begin{align*}
H^0_1 := \spann \{ \Lambda_h(x) \ ; \ \supp(x) \subset E_3  \}, \ \
H^0_2 := \spann \{ \Lambda_h(x) \ ; \ \supp(x) \subset G_1 \}.
\end{align*}
By definition we have $H_1 = \overline{H^0_1}^{\| \cdot \|_2}$ and $H_2 = \overline{H^0_2}^{\| \cdot \|_2}$. Moreover, for a word $t \in M$ we set $d_t = \dim \omega(t)$. Note that, up to this point, we have considered the $d_t$-dimensional corepresentation $\omega(t)$ as an element in $ \Pol(\mathbb G \wr_* S_N^+) \otimes M_{d_t}(\mathbb C)$. However, since we need to distinguish the representation spaces of different corepresentations, we will consider $\omega(t)$ as an element in $ \Pol(\mathbb G \wr_* S_N^+) \otimes B(H_t)$ where $H_t$ is an $d_t$-dimensional Hilbert space. We also put
$ K_t := L^2(B(H_t), \tfrac{1}{d_t} \Tr(\cdot)), $
the GNS-space of $B(H_t)$ with respect to the normalized trace $\frac{1}{d_t} \Tr(\cdot)$. Note that $\omega(t)$ can act on $K_t$ by left or right multiplication. Hence, we can now define a linear map
\[ v_t: H \to H \otimes K_t, \ \ a \mapsto \omega(t)(a \otimes 1) \omega(t)^*. \]
The map $v_t$ is an isometry since the embedding $H \to H \otimes K_t, \ a \mapsto a \otimes 1 $ is isometric and $\omega(t)$ is unitary. We will denote the norm on $L^2(\mathbb G \wr_* S_N^+)$ by $\| \cdot \|_2$ and the correponding inner product by $ \langle \cdot, \cdot \rangle_2$. 

\begin{prop} \label{Hilfsprop3}
Let $ 1_{\mathbb G} \neq \alpha \in Irr (\mathbb G)$. For all $z \in \mathcal S$, we have
\[ \| z \|_2 \leq 14 \max \{ \| z\otimes 1 - v_{(\alpha)}z \|_{H \otimes K_{(\alpha)}}, \| z \otimes 1 - v_{1_{\mathbb G}^2}z \|_{H \otimes K_{1_{\mathbb G}^2}}, \| z \otimes 1 - v_{1_{\mathbb G}^4}z \|_{H \otimes K_{1_{\mathbb G}^4}}  \}. \]
\end{prop}

\begin{proof}
The proof of this result is exactly the same as in \cite{Lem2}. One only has to use our Lemma \ref{hilfslem1} whenever the author his result \cite[Lemma 3.8]{Lem2}.
\end{proof}
The following corollary concludes the proof of Theorem \ref{maintheorem}.
\begin{cor}
Let $(z_n)_{n \in \mathbb N}$ be a bounded asymptotically central sequence in $\mathcal S$. Then, $(z_n)_{n \in \mathbb N}$ is asymptotically trivial. In particular, $L^{\infty}(\mathbb G \wr_* S_N^+)$ is full for any $N \geq 8$ and any CMQG $\mathbb G$ of Kac type.
\end{cor}

\begin{proof}
Let $(z_n)_{n \in \mathbb N}$ be a bounded asymptotically central sequence in $\mathcal S$, i.e.
\[ \| z_n a - a z_n \|_2 \to 0 \ \ \text{for all} \ a \in L^{\infty}(\mathbb G \wr_* S_N^+). \]
We may assume $h(z_n) = 0$ for all $n \in \mathbb N$ since otherwise we can replace $z_n$ by $z_n - h(z_n)1$. Moreover, let $ 1_{\mathbb G} \neq \alpha \in Irr (\mathbb G)$. We have
\[
\| z_n\otimes 1 - v_{(\alpha)}z_n \|_{H \otimes K_{(\alpha)}} = \| (z_n\otimes 1) \omega((\alpha)) - \omega((\alpha))z_n \|_{H \otimes K_{(\alpha)}} \to 0, \]
and similarly we obtain
$\| z_n\otimes 1 - v_{1_{\mathbb G}^2}z_n \|_{H \otimes K_{1_{\mathbb G}^2}} \to 0, \
\| z_n\otimes 1 - v_{1_{\mathbb G}^4}z_n \|_{H \otimes K_{1_{\mathbb G}^4}}   \to 0,$
whenever $n \to \infty$.
Hence, by Proposition \ref{Hilfsprop3} it follows that $\| z_n \|_2 \to 0 \ (n \to \infty)$, i.e. $(z_n)_{n \in \mathbb N}$ is asymptotically trivial. By Proposition \ref{Hilfsprop2}, fullness of $L^{\infty}(\mathbb G \wr_* S_N^+)$ follows.
\end{proof}

\end{document}